\documentclass[review]{elsarticle}
\usepackage{amsmath}
\usepackage{amsfonts}
\usepackage{amsfonts}
\usepackage{amssymb,latexsym}
\usepackage{amsthm}
\usepackage{lineno,hyperref}

\newtheorem{theorem}{\bf Theorem}[section]

\newtheorem{remark}{\bf Remark}[section]
\newtheorem{proposition}{\bf Proposition}[section]
\newtheorem{lemma}{\bf Lemma}[section]

\usepackage{graphicx}
\usepackage{chngcntr}
\counterwithin{figure}{section}
\counterwithin{equation}{section}

\begin{document}

\begin{frontmatter}

\title{On symmetry and recovery of steady continuous stratified periodic water waves\tnoteref{mytitlenote}}
\tnotetext[mytitlenote]{This work was supported in part by NSFC(No.11571057).}

\author{Fei Xu}
\author{Yong Zhang}
\author{Fengquan Li\fnref{*}}

\fntext[*]{E-mail address: fqli@dlut.edu.cn}

\address{Dalian University of Technology, Dalian 116024,
People's Republic of China}

\begin{abstract}
This paper considers two-dimensional steady continuous stratified periodic water waves. Firstly, we prove that each streamline must be symmetric about the crest line when it is strictly monotonous between successive troughs and crests by exploiting the maximum principle and analysis of surface profile. Then, standard Schauder estimates are exploited on the uniform oblique derivative problems to show that all streamlines are real analytic (including the free surface). Based on above symmetry and regularity of streamlines, finally we provide an analytic expansion method to recover the water waves from horizontal velocity on the axis of symmetry and wave height. Most notably, all of results here are suitable not only for small amplitude but also for large amplitude.
\end{abstract}


\end{frontmatter}

\section{\bf Introduction}
The density stratification of ocean dynamics is common due to the change of water temperature and salinity at different depth. Moreover, a heterogeneous density distribution can be roughly divided into two cases. One is called discontinuous stratification, that is to say, there is a thin transition layer called pycnocline between two liquids or one liquid with different density. Thus, as it passes through the pycnocline, the density experiences something close to a jump discontinuity. The other is called continuous stratification, even analytic stratification(see \cite{15}), meaning that the change of density is continuous. Stratification of water waves, a very applied problem, has attracted a great deal of scholarly interests, especially in the oceanography and geophysical fluid dynamics communities.

The existence theories for two-dimensional stratified steady periodic gravity waves with or without surface tension were investigated in \cite{21,22} or in \cite{1}. \cite{14} provided the existence and qualitative theory for stratified solitary water waves. In \cite{13}, the author developed some important results about continuous dependence on the density for stratified steady water waves. At the same time, the symmetry of the monotonous stratified wave of small amplitude has been dealt with in \cite{3}. In fact, we note that the symmetric waves of large amplitude are known to exist in \cite{1}. Thus, it is reasonable to expect that the continuous stratified waves with large amplitude are necessarily symmetric. Another kind of important question, with more practical applications, is about the recovery or determination of the surface waves from some given data. For example, \cite{16,17,18} are about the recovery of steady periodic wave profiles from pressure measurements at the bed. \cite{19} and \cite{20} provided some new methods to determine solitary water wave profiles from pressure transfer functions. \cite{23} derived a power series reconstruction formula from horizontal velocity on vertical symmetry axis for Stokes waves. However, as far as we know, there is less research for the determination of stratified steady water waves than steady water waves. Recently, in \cite{15}, R.M. Chen and S. Walsh studied that the pressure data uniquely determines the solitary stratified wave, both in the real analytic and Sobolev regimes. In fact, the pressure measurements at the flat bed is not easy or precise due to the complex condition in sea and defects of our measuring tools. Thus, we may reasonably ask: Can we recover more information about stratified steady waves from less given data?

Despite some significant investigations on stratified water waves have been carried out in the last decade, other fundamental questions are still remaining to be resolved, just like mentioned above. Therefore, an important aim of this paper is to prove that the strong symmetry result which exist for gravity water waves in \cite{5,6} can be extended to the continuous stratified water waves. Our results are suitable for not only small amplitude but also large amplitude, which is a significant improvement to \cite{3}. The other is on the recovery of wave profile, the velocity field and the pressure distribution in continuous stratified periodic water waves. The proof of our result is mainly based on the real analyticity of each streamline, originating from \cite{8,9,11,12}.

This paper is organized as follows. In Section 2 we will show the equivalent formulations of the government equations for steady stratified water waves. The fact that every monotonous streamline is necessarily symmetric without the restriction of the elevation will be proved in Section 3. At last, the standard Schauder estimates (following\cite{9}) are applied to prove the analyticity of each streamline, then, we provide an analytic expansion method to recover the stratified wave from horizontal velocity on the axis of symmetry and wave height.

\section{\bf Equivalent formulations of stratified water waves}

\subsection{Governing equation in velocity formulation}
Let us now briefly review the setup for 2-D steady continuous stratified periodic water waves in \cite{1}. Fix a Cartesian coordinate system such that the $X$-axis points to be horizontal, and the $Y$-axis to be vertical. We assume that the floor of the sea bed is flat and occurs at $Y=-d$, and $Y=\eta(t,X)$ be the free surface. We shall normalize $\eta$ by choosing the axis so that the free surface is oscillating around the line $Y=0$. As usual we let $u=u(t,X,Y)$ and  $v=v(t,X,Y)$ denote the horizontal and vertical velocities respectively, and let $\rho=\rho(t,X,Y)> 0$ be the density and $P=P(t,X,Y)$ be the pressure, all of which have the form $(X-ct)$ due to considering steady travelling wave in this paper where $c$ represents the speed of wave. For convenience, we denote $x=X-ct, y=Y$ in following states and consider problem in $\Omega=\{(x,y)|-\pi<x<\pi, -d<y<\eta(x)\}$ due to the periodicity.

For water waves, it is appropriate to suppose that the flow is incompressible. Mathematically, this assumption manifests itself as the requirement that the vector field be divergence free for all time
\begin{eqnarray}
u_{X}+v_{Y}=0. \label{eq2.1}
\end{eqnarray}
Taking the fluid to be inviscid, conservation of mass implies
\begin{eqnarray}
\rho_{t}+(\rho u)_{X}+(\rho v)_{Y}=0. \label{eq2.2}
\end{eqnarray}
However, the relation of time-space $(X-ct)$ and (\ref{eq2.1})(\ref{eq2.2}) demonstrate
\begin{eqnarray}
(u-c)\rho_{x}+v\rho_{y}=0. \label{eq2.3}
\end{eqnarray}
Therefore, the governing equations in velocity field formulation are expressed by the nonlinear free-boundary problem (see\cite{1})
\begin{eqnarray}
\left\{\begin{array}{llll}{u_{x}+v_{y}=0} & {\text { for }-d \leq y \leq \eta_(x)} \\
{(u-c)\rho_{x}+v\rho_{y}=0} & {\text { for }-d \leq y \leq \eta_(x)} \\
{\rho(u-c) u_{x}+ \rho v u_{y}=-P_{x}} & {\text { for }-d \leq y\leq \eta(x)} \\
{\rho(u-c) v_{x}+\rho v v_{y}=-P_{y}-g\rho} & { \text { for }-d \leq y \leq \eta(x)} \\
{v=0} & {\text { for } y=-d} \\
{v=(u-c) \eta_{x}} & {\text { on } y=\eta(x)} \\
{P=P_{atm}} & {\text { on } y=\eta(x)}\end{array}\right. \label{eq2.4}
\end{eqnarray}
where $P_{atm}$ is the constant atmosphere pressure, and $g=9.8 m/s^{2}$ is the (constant) gravitational acceleration at the Earth's surface. The solutions we consider are periodic in the variable $x$, namely $(u,v,P,\eta)$ are all $2\pi$-periodic in $x$, and the stagnation points are excluded from the flow. The latter property is satisfied if we assume that $u < c$ throughout the fluid.

Recall that we have chosen our axis so that $\eta$ oscillates around the line $y=0$. In other words,
We shall also denote
\begin{eqnarray}
\int^{\pi}_{-\pi}\eta(x)dx=0. \label{eq2.5}
\end{eqnarray}
At the same time, we let
\begin{eqnarray}
\eta_{max}:=\mathop{max}\limits_{[-\pi,\pi]}\eta(x)+d,~~~~\eta_{min}:= \mathop{min}\limits_{[-\pi,\pi]}\eta(x)+d. \label{eq2.6}
\end{eqnarray}
These are the maximum and minimum distances between the surface and the bed respectively. In fact, we will see that $\eta_{max}=\eta(0)+d$ and $\eta_{min}=\eta(\pm\pi)+d$, where $\eta(0)$ is defined as wave height later.

\subsection{Governing equation in stream function formulation}
Observe that, by conservation of mass and incompressibility, $\rho$ is transported and the vector field is divengence free. Therefore we may introduce a (relative) pseudo-stream function $\psi=\psi(x,y)$ satisfying
\begin{eqnarray}
\psi_{x}=-\sqrt {\rho}v, ~~~~~~~\psi_{y}=\sqrt {\rho}(u-c)<0. \label{eq2.7}
\end{eqnarray}
That is to say, here we add a $\rho$ factor to the typical definition of the stream function for an incompressible fluid (see \cite{7}).

It is a straightforward calculation to check that $\psi$ is indeed a (relative)stream function in the usual sense, i.e. its gradient is orthogonal to the vector field in the moving frame at each point in the fluid domain. As usual, we shall refer to the level sets of $\psi$ as the streamlines of the flow. For definition we choose $\psi\equiv 0$ on the free boundary $y=\eta(x)$, so that $\psi\equiv -p_{0}>0$ on $y=-d$ due to our assumption $u<c$, where
\begin{eqnarray}
p_{0}= \int_{-d}^{\eta(x)}\sqrt {\rho(x,y)}[u(x,y)-c]dy, \label{eq2.8}
\end{eqnarray}
is called pseudo mass flux. Since $\rho$ is transported, it must be constant on the streamlines. We may therefore let streamline density function $\rho\in C^{1,\alpha}([p_{0}, 0]; R^{+})$ be given such that
\begin{eqnarray}
\rho(x,y)=\rho(-\psi(x,y)). \label{eq2.9}
\end{eqnarray}
throughout the fluid. Moreover, the streamline density function $\rho$ is nonincreasing, meaning that $\rho'\leq0$.

From Bernoulli's law, we know that
\begin{eqnarray}
E=\frac{\rho}{2}((u-c)^{2}+v^{2})+gy\rho+P \label{eq2.10}
\end{eqnarray}
is a constant along each streamline. Then, under the assumption that $u < c$ throughout the fluid, there exists a function $\beta \in C^{1,\alpha}([0, |p_{0}|]; R)$ such that
\begin{eqnarray}
\frac{dE}{d\psi}=-\beta(\psi),  \label{eq2.11}
\end{eqnarray}
where $\beta$ is called the Bernoulli function corresponding to the flow (see \cite{1}). Physically it describes the variation of specific energy as a function of the streamlines. It is worth noting that when $\rho$ is a constant, $\beta$ reduces to the vorticity function. Thus, (\ref{eq2.10}) and (\ref{eq2.11}) show that
\begin{eqnarray}
\frac{dE}{d\psi}=\Delta\psi-gy\rho'(-\psi)=-\beta(\psi(x,y)). \label{eq2.12}
\end{eqnarray}
Moreover, evaluating Bernoulli's theorem on the free surface $\psi \equiv 0$, we find
\begin{eqnarray}
2E|_{\eta}=2P_{atm}+|\nabla\psi|^{2}+2g\eta\rho(-\psi)~~~~~~on ~~y=\eta(x). \label{eq2.13}
\end{eqnarray}

Summarizing the above considerations, we can reformulate the governing equations as the free boundary problem:
\begin{eqnarray}
\left\{\begin{array}{ll}{ \Delta\psi-gy\rho'(-\psi)=-\beta(\psi)} & {\text { in }-d < y<\eta(x)} \\
{|\nabla\psi|^{2}+2g\rho(-\psi)(y+d)=Q} & { \text { on } y=\eta(x)} \\
{\psi=0} & {\text { on } y=\eta(x)} \\
{\psi=-p_{0}} & {\text { on } y=-d}\end{array}\right. \label{eq2.14}
\end{eqnarray}
where $Q=2(E|_{\eta}-P_{atm}+g\rho |_{\eta}d)$.

\subsection{Governing equation in height function formulation}
Now we consider the alternative system $(q,p)$, where
\begin{eqnarray}
q=x,~~~~~p=-\psi(x,y), \label{eq2.15}
\end{eqnarray}
originating from Dubreil-Jacotin's transformation in \cite{24}, which transforms the fluid domain
$$\Omega=\{(x,y):~x\in(-\pi,\pi),~-d<y<\eta(x)\}$$
into rectangular domain
$$D=\{(q,p):~-\pi< q < \pi,~ p_{0}< p< 0\},$$
with $$\overline{D}=\{(q,p):~-\pi\leq q \leq \pi,~ p_{0}\leq p\leq 0\}.$$
Given this, we shall denote
$$ T:=\{(q,p)\in D : p=0\}, ~B=\{(q,p)\in D :p=p_{0}\}.$$
Note that, in light of (\ref{eq2.9}) and (\ref{eq2.11}), we have that
\begin{eqnarray}
\beta=\beta(-p), \rho=\rho(p). \label{eq2.16}
\end{eqnarray}
Next we define
\begin{eqnarray}
h(q,p):=y+d, \label{eq2.17}
\end{eqnarray}
which gives the height above the flat bottom. Some simple calculations imply
\begin{eqnarray}
\psi_{y}=-\frac{1}{h_{p}},~~~~ \psi_{x}=\frac{h_{q}}{h_{p}}, \label{eq2.18}
\end{eqnarray}
\begin{eqnarray}
\partial_{q}=\partial_{x}+h_{q}\partial_{y},~~~~ \partial_{p}=h_{p}\partial_{y}. \label{eq2.19}
\end{eqnarray}
We have normalized $\eta$ so that it has mean zero. Taking the mean of (\ref{eq2.17}) along $T$, We obtain
\begin{eqnarray}
d=\int_{-\pi}^{\pi}h(q,0)dq, \label{eq2.20}
\end{eqnarray}
which is the average value of $h$ over $T$. We also have
\begin{eqnarray}
h_{q}=\frac{v}{u-c},~~~~ h_{p}=\frac{1}{\sqrt{\rho}(c-u)}>0, \label{eq2.21}
\end{eqnarray}
\begin{eqnarray}
u=c-\frac{1}{\sqrt{\rho}h_{p}},~~~~ v=-\frac{h_{q}}{\sqrt{\rho}h_{p}}, \label{eq2.22}
\end{eqnarray}
Consequently, we can rewrite the governing equations as height function formulation:
\begin{eqnarray}
\left\{\begin{array}{lll}{\left(1+h_{q}^{2}\right) h_{p p}-2 h_{q} h_{p} h_{q p}+h_{p}^{2} h_{q q}+[\beta-g(h-d)\rho']h_{p}^{3}=0}  & {\text { in } p_{0}<p<0} \\
{1+h_{q}^{2}+h_{p}^{2}(2g\rho h-Q)=0} & {\text { on } p=0} \\
{h=0} & {\text { on } p=p_{0}}\end{array}\right. \label{eq2.23}
\end{eqnarray}
Now let us define the nonlinear differential operators
\begin{eqnarray}
A(h)\varphi=(1+h_{q}^{2})\varphi_{pp}-2h_{q}h_{p}\varphi_{qp}+h_{p}^{2}\varphi_{qq}. \label{eq2.24}
\end{eqnarray}
\begin{eqnarray}
B(h)\varphi=h_{q}\varphi_{q}+(2gph-Q)h_{p}\varphi_{p}+g\rho h_{p}^{2}\varphi. \label{eq2.25}
\end{eqnarray}
In fact, the operator $A(h)$ is uniformly elliptic due to
$$
\Delta=(h_{q}h_{p})^{2}-(1+h_{q}^{2})h_{p}^{2}=-h_{p}^{2}\leq-min_{\overline{D}}h_{p}^{2}<-m<0,
$$
and the operator $B(h)$ is uniformly oblique due to
$$
(2g\rho h-Q)h_{p}=-\frac{1+h_{q}^{2}}{h_{p}}\leq-\frac{1}{max_{\overline{D}}h_{p}}\leq-M<0.
$$
The periodicity would ensure that the boundary conditions on the sides $q=\pm\pi$ of $D$ are
\begin{eqnarray}
\left\{\begin{array}{ll}{h(\pm\pi,0)<h(q,0)}  & {\text { for } -\pi<q<\pi} \\
{h(\pm\pi,p)\leq h(q,p)} & {\text { for } -\pi<q<\pi, p_{0}\leq p<0}\end{array}\right. \label{eq2.26}
\end{eqnarray}
which express, respectively, that $q=\pm\pi$ is the single wave trough and that every streamline attains a minimum below the trough. In addition, every streamline is required to be monotone near the trough, i.e.
\begin{eqnarray}
h(q_{1},p)\leq h(q_{2},p)~~~~for -\pi<q_{1}<q_{2}<-\pi+\varepsilon,~p_{0}\leq p\leq 0,  \label{eq2.27}
\end{eqnarray}
where $\varepsilon>0$ is small.

\section{\bf Symmetry of each streamline in stratified water waves}
In this section, we are going to prove the following Theorem \ref{theorem3.1} whose proof is based on elliptic maximum principles and the moving plane method, cf.\cite{2,4}. The moving plane method mainly consists in its setup and moving the line until an extremal position is reached. In fact, (\ref{eq2.27}) can ensure the setup of moving plane method for our problem. Then, using sharp elliptic maximum principles, (cf. Lemma \ref{lemma3.1} and Lemma \ref{lemma3.2}), we show that the limiting line is the crest line $x=0$ and that each streamline is symmetric with respect to it.

\begin{theorem}\label{theorem3.1}
Let $h\in C^{2,\alpha}(\overline{D})$ be the solution to (\ref{eq2.23}) when the streamline density function $\rho \in C^{1,\alpha}([p_{0},0];R^{+})$ and Bernoulli function $\beta \in C^{0,\alpha}([0,|p_{0}|];R)$ are given. Assume each streamline is monotone between successive crest line and through line with period $2\pi$, then it is symmetric about the crest line $x=0$.
\end{theorem}

\begin{remark}\label{remark3.1}
The existence of solution $h$ to (\ref{eq2.23}) is strictly proved in \cite{1}.
\end{remark}

Before starting our proof, we firstly introduce some useful Lemmas about regional maximal principle.

\begin{lemma} \label{lemma3.1}
(see Theorem 2.13, Theorem 2.15 and Remark 2.16 in \cite{2}) \\
Let $D\subset R^{2}$ be an open rectangle and $\omega\in C^{2}(\overline{D})$ satisfy $L\omega\leq0$ for some uniformly elliptic operator $L=a_{ij}\partial_{ij}+b_{i}\partial_{i}+c$ with measurable and bounded coefficients in $\overline{D}$, moreover, such that $\inf_{D}\omega=0$. Then the followings hold (without condition $c\leq 0$ in $D$)\\
$(1)$ The weak maximum principle: $\omega$ attains its minimum on $\partial D$. \\
$(2)$ The strong maximum principle: If $\omega$ attains its minimum in $D$, then $\omega$ is constant in $D$.\\
$(3)$ Hopf's maximum principle: Let Q be a point on $\partial D$, different from the corners of the rectangle $\overline{D}$. If $\omega(Q)<\omega(X)$ for all $X$ in $D$, then $\partial_{\nu}\omega(Q)\neq 0$.
\end{lemma}

\begin{lemma} \label{lemma3.2}
(Serrin's Edge-point lemma, see Lemma 2 in \cite{4})
Let
$$
D: = \{(x,z)\in R^{2}: a< x < b, m < z < \xi(x)\},
$$
where $a<b$, $\xi \in C^{2}([a,b])$, $\xi > m$, and $\xi^{'}(a)=0$ $(or~\xi^{'}(b)=0)$. Let further $\omega\in C^{2}(\overline{D})$ satisfies $L\omega\leq0$ in $D$ for some uniformly elliptic operator $L=a_{ij}\partial_{ij}+b_{i}\partial_{i}$ with measurable and $L^{\infty}$ coefficients in $\overline{D}$. If the edge point $Q=(a, \xi(a))$ $(or~Q=(b,\xi(b))$ satisfies $\omega(Q)=0$ and $\omega\geq0$ in $D$, then either
$$
\frac{\partial \omega}{\partial s} > 0~~ or ~~\frac{\partial^{2} \omega}{\partial ^{2}s} > 0 ~~at ~Q,
$$
where $s\in R^{2}$ is any direction at $Q$ that enters $D$ non-tangentially.
\end{lemma}

\begin{remark} \label{remark3.2}
Suppose for now that $\omega$ has a determined sign in some region $D'\subset D$ whose boundary includes the corner point $Q$ and which is blunt in the sense that the Serrin edge lemma holds here for $L=a_{ij}\partial_{ij}+b_{i}\partial_{i}+c$ with $L^{\infty}$ coefficients (see Definition E3 of \cite{2}).
\end{remark}

\begin{lemma} \label{lemma3.3}
(see Proposition 1.1 in \cite{26})
Suppose that \\
(1)$D$ is bounded, $\omega\in C^{2}(\overline{D})$;\\
(2)$L\omega=(a_{ij}\partial_{ij}+b_{i}\partial_{i}+c)\omega \geq 0$ $(\leq0)$ where $L$ is uniformly elliptic operator with measurable and $L^{\infty}$ coefficients;\\
(3)$u\mid_{\partial\Omega}\leq 0$ $(\geq0)$;\\
Then $u\leq 0$ $(\geq0)$ on $\overline{D}$, whenever $|D|< \delta$, where the positive number $\delta$ is independent of $u$ and $D$ (but depend on diam $D$ and $\|c\|_{L^{\infty}}$).
\end{lemma}

Now we turn to the proof of our symmetry result.
\begin{proof}
Find that $h$, $\widetilde{h}$ are solutions to (\ref{eq2.23}) and define
\begin{eqnarray}
L:=A(h)+B_{1}(h,\widetilde{h})\partial_{q}+B_{2}(h,\widetilde{h})\partial_{p}+C, \label{eq3.1}
\end{eqnarray}
where the operator $A(h)$ is defined by (\ref{eq2.24}) and other operators are given by
\begin{eqnarray}
\left\{\begin{array}{lll}
{B_{1}(h,\widetilde{h})=\widetilde{h}_{pp}(h_{q}+\widetilde{h}_{q})-2h_{p}\widetilde{h}_{qp}}, \\
{B_{2}(h,\widetilde{h})=\widetilde{h}_{qq}(h_{p}+\widetilde{h}_{p})-2\widetilde{h}_{q}\widetilde{h}_{qp}+
[\beta-g\rho'(\widetilde{h}-d)](h_{p}^{2}+h_{p}\widetilde{h}_{p}+\widetilde{h}_{p}^{2})}, \\
{C=-g\rho'h_{p}^{3}}.\end{array}\right. \label{eq3.3}
\end{eqnarray}
Thus, $L$ is uniformly elliptic operator of second order.
Let $\nu=\widetilde{h}-h$, then it's not difficult for us to verify that $\nu$ satisfies following equations due to (\ref{eq2.23})
\begin{eqnarray}
\left\{\begin{array}{lll}
{L\nu=0} & {\text { in } p_{0}<p<0},\\
{(h_{q}+\widetilde{h}_{q})\nu_{q}+(h_{p}+\widetilde{h}_{p})(2\rho g\widetilde{h}-Q)\nu_{p}+2h_{p}^{2}\rho g\nu =0} & {\text { on } p=0}, \\
{\nu=0} & {\text { on } p=p_{0}}. \end{array}\right. \label{eq3.3}
\end{eqnarray}
For a reflection parameter $\lambda \in (-\pi, 0)$, the reflection of $q$ about $\lambda$ is given by
$q^{\lambda}= 2\lambda-q$, let us define
 $$ D^{\lambda}=\{(q,p)\in D: -\pi<q<\lambda,~p_{0}<p<0\}.$$
At $(q,p)\in D^{\lambda}$, the associated reflection function is
$$ \omega(q,p;\lambda)=h(2\lambda-q,p)-h(q,p),$$
which satisfies $L\omega=0$. At the same time, the reflection function $\omega(\cdot,\cdot;\lambda)$ also satisfies the boundary condition
\begin{eqnarray}
\left\{\begin{array}{ll}
{\omega(\lambda,p;\lambda)=0} &~for~~ {p\in[p_{0}, 0]}, \\
{\omega(q,p_{0};\lambda)=0} & ~for~~{q\in[-\pi, \lambda]},\\
{\omega(-\pi,p;\lambda)=0} &~for~~ {p\in[p_{0}, 0]}.
\end{array}\right.\label{eq3.4}
\end{eqnarray}
The first property is immediate from the definition of $\omega(q,p;\lambda)$, the second follows from the boundary condition $h=0$ on $p=p_{0}$ and the third follows from (\ref{eq2.26}). $\omega(q,0;\lambda)\geq 0$ for $ -\pi<q<-\pi+\epsilon$ where $\epsilon > 0$ is as in (\ref{eq2.27}). Hence, for $0<\lambda+\pi<\epsilon$ with $\epsilon>0$ small, $\omega\geq 0$ on $D^{\lambda}$. Let
$$\lambda_{0}=\sup \{\lambda\in(-\pi,0]:\omega(q,p;\lambda)\geq 0~ in~ D^{\lambda}\}.$$
One of the following two cases can occur:\\
{\bf Case $1$}: $\lambda_{0}=0$;\\
{\bf Case $2$}: $\lambda_{0}\in (-\pi,0).$\\
{\bf Step $1$}: As $\lambda_{0}=0$, according to the definition of $\lambda_{0}$ and (\ref{eq3.4}), the reflection function yield the following boundary conditions
\begin{eqnarray}
\left\{\begin{array}{llll}
{\omega(-\pi,p;0)\geq0} &~for~~ {p\in[p_{0}, 0]}, \\
{\omega(0,p;0)= 0} & ~for~~{p\in[p_{0}, 0]}, \\
{\omega(q,0;0)\geq 0} &~for~~ {q\in[-\pi, 0]},  \\
{\omega(q,p_{0};0)= 0} &~for~~ {q\in[-\pi, 0]}
\end{array}\right.\label{eq3.5}
\end{eqnarray}
Consider the rectangle $D^{0}\doteq(-\pi , 0)\times(p_{0},0)$. Since $inf_{D^{0}}\omega=0$ due to the definition of $\lambda_{0}$, the strong maximum principle in Lemma \ref{lemma3.1} implies
\begin{eqnarray}
\omega(q,p;0)>0 ~~ \text{in}~ D^{0} ~~~~\text{or} ~~~~\omega(q,p;0) \equiv 0 ~~ \text {in} ~D^{0}. \label{eq3.6}
\end{eqnarray}
If $\omega$ vanishes throughout $D^{0}$, we have symmetry.

Next, what we need to do is preclude the case \{$\omega(q,p;0)>0$ in $D^{0}$\}. Thus, we assume $\omega(q,p;0)>0$. At the the trough $(-\pi,0)$, we have
$$
\omega(-\pi,0;0)=\omega_{p}(-\pi,0;0)=\omega_{pp}(-\pi,0;0)=\omega_{q}(-\pi,0;0)=\omega_{qq}(-\pi,0;0)=0
$$
due to the evenness of $h$, and the periodicity of $h$ implies
$$
\omega(-\pi,0;0)=\omega_{q}(-\pi,0;0)=\omega_{qq}(-\pi,0;0)=0.
$$
While differentiating the second equation of (\ref{eq2.23}) with respect to $q$, we have
$$2h_{q}h_{qq}+ 2h_{p}h_{qp}(2\rho gh-Q)+ 2g\rho h_{p}^{2}h_{q}=0,$$
which forces $ h_{qp}(-\pi,0)=0$, since $h_{p}>0$ and $2\rho gh-Q<0$. According to the same boundary condition, a similar argument holds for the reflection $h(-q,p)$, whence $\omega_{qp}(-\pi, 0)=0$. Since $\omega$ has a determined sign in region $D^{0}$ whose boundary includes the corner point $(-\pi, 0)$ and which is blunt in the sense that the Serrin edge lemma (Remark \ref{remark3.2}) holds there. This is a contradiction, which means that the wave is symmetric about the crest located at $q=0$.\\
{\bf Step $2$}: As $\lambda_{0}\in (-\pi, 0)$, (\ref{eq3.4}) and the definition of $\lambda_{0}$ imply
\begin{eqnarray}
\left\{\begin{array}{llll}
{\omega(-\pi,p;\lambda_{0})\geq 0} &~for~~ {p\in[p_{0}, 0]}, \\
{\omega(\lambda_{0},p;\lambda_{0})=0} &~for~~ {p\in[p_{0}, 0]},\\
{\omega(q,0; \lambda_{0})\geq0} &~for~~ {q\in[-\pi,\lambda_{0}]},\\
{\omega(q_{0},p_{0}; \lambda_{0})= 0} & ~for~~{q\in[-\pi,\lambda_{0}].}
\end{array}\right.\label{eq3.7}
\end{eqnarray}
If we redefine $D^{\lambda_{0}}=(-\pi, \lambda_{0})\times (p_{0}, 0)$, $\omega(q,p; \lambda_{0})\geq 0$ on the boundary of $D^{\lambda_{0}}$. It follows from the strong maximum principle in Lemma \ref{lemma3.1} (the predetermined sign of $\omega$) that
\begin{eqnarray}
\omega(q,p;\lambda_{0})>0 ~~ \text{in}~ D^{\lambda_{0}} ~~~~\text{or} ~~~~\omega(q,p;\lambda_{0}) \equiv 0~~ \text {in} ~D^{\lambda_{0}}. \label{eq3.8}
\end{eqnarray}

Let us firstly assume that
\begin{eqnarray}
\omega(q,p;\lambda_{0})>0 & in~ D^{\lambda_{0}}, \label{eq3.9}
\end{eqnarray}
which will be precluded according to Narrow region Lemma \ref{lemma3.3}. Let $\delta_{0}= \min\{-\lambda_{0}, \delta\}$ ($\delta $ is determined below), and the function $\omega(q,p;\lambda_{0}+\delta_{0})$, denoting as $\omega^{\lambda_{0}+ \delta_{0}}$, on the narrow region $D_{\delta_{0}}=D^{\lambda_{0}+\delta_{0}}/D^{\lambda_{0}-\delta_{0}}$ will be studied. It's easy to obtain $\omega(q,p; \lambda_{0}+\delta_{0})$ satisfing $L\omega^{\lambda_{0}+\delta_{0}}=0$ and following boundary conditions
\begin{eqnarray}
\left\{\begin{array}{lll}
{\omega(\lambda_{0}+\delta_{0},p;\lambda_{0}+\delta_{0})=0} &~for~~ {p\in[p_{0}, 0]}, \\
{\omega(\lambda_{0}-\delta,p;\lambda_{0}+\delta_{0})\geq0} &~for~~  {p\in[p_{0}, 0]}, \\
{\omega(q,p_{0};\lambda_{0}+\delta_{0})=0} &~for~~  {q\in[-\pi,\lambda_{0}+\delta_{0}]}.
\end{array}\right.\label{eq3.10}
\end{eqnarray}
The first property is immediate from the definition of $\omega$, the second follows from (\ref{eq3.9}) and the definition of $\lambda_{0}$, the last one follows from the boundary condition $h=0$ on $p=p_{0}$. Then we will show $\omega(q,0;\lambda_{0}+\delta_{0})\geq 0$ as $q\in[\lambda_{0}-\delta_{0}, \lambda_{0}+\delta_{0}]$.
Consider the top boundary condition, the function $\omega^{\lambda_{0}}(q, 0)> 0$ for $q\in(-\pi, \lambda_{0})$; Otherwise exist $q_{0}\in(-\pi, \lambda_{0})$ such that $\omega^{\lambda_{0}}(q_{0}, 0)=0$, and $\omega^{\lambda_{0}}_{q}(q_{0}, 0)=0$. The boundary condition $1+h_{q}^{2}+h_{p}^{2}(2\rho gh-Q)=0$ where $\rho$ is constant along the streamline, forces $h_{p}(q_{0}, 0)=h_{p}(2\lambda_{0}-q_{0}, 0)$, which is in contradiction with the Hopf lemma  in Lemma \ref{lemma3.1}. Since $\omega^{\lambda_{0}}(q,0)>0$ for $q\in (-\pi, \lambda_{0})$ and $\omega^{\lambda_{0}}(\lambda_{0},0)=0$, we have
\begin{eqnarray}
\partial_{q}\omega^{\lambda_{0}}(\lambda_{0},0)<0. \label{eq3.11}
\end{eqnarray}
Since $\partial_{q}\omega^{\lambda}(q,0)$ as a continuous function of $q$ and $\lambda$, statement (\ref{eq3.11}) implies that for some enough small $\delta$,
$$\partial_{q}\omega^{\lambda}(q,0)<0$$
provided $|\lambda-\lambda_{0}|<\delta$ and $|q-\lambda_{0}|<\delta$. Thus $\omega^{\lambda}(\lambda,0)=0$ and  $\omega^{\lambda}(q,0)>0$ for any $\lambda, q$ such that $\lambda_{0}-\delta\leq q < \lambda \leq\lambda_{0}+\delta$, we can get $\omega^{\lambda_{0}+\delta_{0}}(q,0)\geq 0$ for $q\in[\lambda_{0}-\delta, \lambda_{0}+\delta]$.
Now we apply the narrow region Lemma \ref{lemma3.3} to derive that $\omega(q,p;\lambda_{0}+\delta_{0})\geq 0$ for $(q,p)\in D_{\delta}$ and $\omega(q,p; \lambda_{0})\geq 0$ for $(q,p)\in D^{\lambda_{0}+\delta_{0}}$.
This is in contradiction with the definition of $\lambda_{0}$.

Thus, we have
\begin{eqnarray}
\omega(q,p;\lambda_{0})\equiv 0~in~~ D^{\lambda_{0}}.
\label{eq3.12}
\end{eqnarray}
Note that as long as $2\lambda_{0}+\pi$ lies to the left of the wave crest, $\omega(q,p ;\lambda_{0})\geq 0$ will hold for $-\pi\leq q \leq \lambda_{0} $ by the monotonicity of the streamline between trough and crest. Therefore,  $2\lambda_{0}+\pi$ lies to the right of wave crest or at least in line with the wave crest, {\bf Case 2} implies that $h(q,p)$ is nonincreasing for $q \in (2\lambda_{0}+\pi, \pi)$, then $h(q,p)\equiv h(\pi, p)$ whenever $2\lambda_{0}+\pi\leq q \leq \pi$, Since $\omega(2\lambda_{0}+\pi,p; \lambda_{0})=0$ yields $h(2\lambda_{0}+\pi,p)=h(-\pi,p)=h(\pi,p)$ and the map $q\mapsto h(q,p)$ is nonincreasing on that interval as we established that $2\lambda_{0}+\pi$ lies to the right of the wave crest. We have $h(q,p)= h(2\lambda_{0}-q,p)$ for all $q\in [-\pi, \lambda_{0}]$ so that $q=\lambda_{0}$ must be the location of the wave crest.
\end{proof}

\section{\bf Recovery of the wave from horizontal velocity on axis of symmetry and wave height}
The recovery theorem mainly depends on to a large extent on the fact that the streamlines are symmetry about crest lines which has been considered in Section 3. At the same time, the regularity of streamlines which will be considered in Lemma \ref{lemma4.2} is also significant for following proof of our result.
\begin{theorem}\label{theorem4.1}
Consider a two-dimensional steady continuous stratified periodic water wave with known analytic stratified density $\rho(-\psi)$ and Benoulli's function $\beta(\psi)$, moreover, assume that the horizontal velocity on the axis of symmetry, i.e. $u(0,y)$, and wave height $\eta(0)$ are given. Then,
we can recover the wave profile, the velocity field and the pressure within the fluid.
\end{theorem}
\begin{remark}
The equations (\ref{eq2.4}) with solutions $(u,v,P,\eta)$ are equivalent to the equations (\ref{eq2.14}) with solutions $(\eta,\psi)$. (see \cite{1}) \label{remark4.1}
\end{remark}
Before proving this theorem, we firstly take the analyticity of steady continuous stratified periodic water wave's streamlines into consideration. Some technique of estimates and standard Schauder estimates will be used to prove regularity of each streamline (including free surface).
\begin{lemma}\label{lemma4.1}
(see Lemma 3.2 in \cite{9})
Let $l=1$ or $2$ be given, and let $\|\cdot\|$ stand for some H$\ddot{o}$lder norm $\|\cdot\|_{0,\alpha}$ or $\|\cdot\|_{1,\alpha}$. Suppose that $k_{0}$ is an integer with $k_{0}\geq l+1$, and $\partial^{k}_{q}u_{j}\in C^{0,\alpha}(\overline{D})$ for all $k\leq k_{0}$, $j=1,2,3$. If there exists a constant $H\geq 1$ such that
$$
\forall l+1\leq k\leq k_{0}, ~~\|\partial^{k}_{q}u_{j}\|\leq H^{k-l}(k-l-1)!, ~~j=1,2,3,
$$
then we can find a constant $C_{*}$ depending only on $l$ such that
$$
\forall l+1\leq k\leq k_{0}, ~~\|\partial^{k}_{q}(u_{1}u_{2}u_{3})\|\leq C_{*}(\sum^{3}_{j=1}\|u_{j}\|_{l+1,\alpha}+1)^{6}H^{k-l}(k-l-1)!.
$$
\end{lemma}
\begin{lemma}\label{lemma4.2}
Let the functions $\rho\in C^{1,\alpha}([p_{0},0];R^{+})$, $\beta\in C^{0,\alpha}([0,|p_{0}|];R)$ be given with $\rho'(p)\leqslant0, p_{0}<0$; and let $h\in C^{2,\alpha}_{per}(\overline{D})$ with $h_{p}>0$ be the solutions to (\ref{eq2.23}). Then, the mapping $q\mapsto h(q,p)$, with any fixed $p\in [p_{0},0]$, is analytic in $R$.
\end{lemma}
\begin{proof}
{\bf Step 1}: We show that $\partial_{q}^{n}h\in C^{2,\alpha}_{per}(\overline{D})$ for any $n\in N$.

Indeed, we take the derivative with respect to $q$ on both sides of (\ref{eq2.23}), then
\begin{eqnarray}
\left\{\begin{array}{lll}{A(h)h_{q}=f_{1}+f_{2}}  & {\text { in } p_{0}<p<0}, \\
{B(h)h_{q}=0} & {\text { on } p=0}, \\
{h_{q}=0} & {\text { on } p=p_{0}.}\end{array}\right. \label{eq4.1}
\end{eqnarray}
where operators $A(h)$ and $B(h)$ are, given by (\ref{eq2.24})(\ref{eq2.25}), uniformly elliptic and uniformly oblique, with $f_{1}$ and $f_{2}$ given by
\begin{eqnarray}
\left\{\begin{array}{ll}
{f_{1}=-(\partial_{q}h_{q}^{2})h_{pp}+2(\partial_{q}(h_{p}h_{q}))h_{qp}-(\partial_{q}h_{p}^{2})h_{qq}}, \\
{f_{2}=-(\beta+gd\rho')(\partial_{q}h_{p}^{3})+g\rho'\partial_{q}(hh_{p}^{3})}.\end{array}\right. \label{eq4.2}
\end{eqnarray}
Since $h\in C^{2,\alpha}_{per}(\overline{D})$, the the coefficients of the operator $A(h)$ and $B(h)$ are in $C^{1,\alpha}_{per}(\overline{D})$. Similarly, $\rho\in C^{1,\alpha}([p_{0},0];R^{+})$ and $\beta\in C^{0,\alpha}([0,|p_{0}|];R)$ will lead that the right-hand side $f_{1}$ and $f_{2}$ are in $C^{0,\alpha}_{per}(\overline{D})$
due to the properties of $H\ddot{o}lder$ space. Therefore, we can use the standard Schauder estimate (see Theorem 6.30 in \cite{10}) to obtain $h_{q}\in C^{2,\alpha}_{per}(\overline{D})$ with
$$
\parallel h_{q}\parallel_{C^{2,\alpha}}\leq C(\parallel h_{q}\parallel_{C^{0}}+\parallel f_{1}\parallel_{C^{0,\alpha}}+\parallel f_{2}\parallel_{C^{0,\alpha}})
$$
and $C=C(\alpha,\parallel h\parallel_{C^{2,\alpha}})$.

Then we repeat this process by taking the derivative with respect to $q$ up to $2$ order, $3$ order until $n$ order gradually on both sides of (\ref{eq2.23}). The Leibniz formula is used to show
\begin{eqnarray}
\left\{\begin{array}{lll}{A(h)[\partial_{q}^{n}h]=F_{1}+F_{2}}  & {\text { in } p_{0}<p<0}, \\
{B(h)[\partial_{q}^{n}h]=\Phi_{1}+\Phi_{2}} & {\text { on } p=0}, \\
{\partial_{q}^{n}h=0} & {\text { on } p=p_{0}.}\end{array}\right. \label{eq4.3}
\end{eqnarray}
where operators $A(h)$ and $B(h)$ are, given by (\ref{eq2.24})(\ref{eq2.25}), uniformly elliptic and uniformly oblique, with $F_{1}$, $F_{2}$, $\Phi_{1}$ and $\Phi_{2}$ given by
\begin{eqnarray}
\left\{\begin{array}{llll}
{F_{1}=\sum_{k=1}^{n}C^{k}_{n}[-(\partial_{q}^{k}h_{q}^{2})(\partial_{q}^{n-k}h_{pp})+
2(\partial_{q}^{k}(h_{p}h_{q}))(\partial_{q}^{n-k}h_{qp})-(\partial_{q}^{k}h_{p}^{2})(\partial_{q}^{n-k}h_{qq})]}, \\
{F_{2}=-(\beta+gd\rho')(\partial_{q}^{n}h_{p}^{3})+g\rho'\sum_{k=0}^{n}C_{n}^{k}(\partial_{q}^{k}h)
(\partial_{q}^{n-k}h_{p}^{3})}, \\
{\Phi_{1}=-\frac{1}{2}\sum_{k=1}^{n-1}C^{k}_{n}(\partial_{q}^{k}h_{q})(\partial_{q}^{n-k}h_{q})-
\frac{1}{2}(2g\rho h-Q)\sum_{k=1}^{n-1}C^{k}_{n}(\partial_{q}^{k}h_{p})(\partial_{q}^{n-k}h_{p})}, \\
{\Phi_{2}=-g\rho\sum_{k=1}^{n-1}C^{k}_{n}(\partial_{q}^{k}h)(\partial_{q}^{n-k}h_{p}^{2})}.\end{array}\right. \label{eq4.4}
\end{eqnarray}
where $C^{k}_{n}=\frac{n!}{k!(n-k)!}$.
Thus we can use the Schauder estimate  based on standard interaction to get $\partial_{q}^{n}h\in C^{2,\alpha}_{per}(\overline{D})$ with
$$
\parallel \partial_{q}^{n}h\parallel_{C^{2,\alpha}}\leq C(\parallel \partial_{q}^{n}h\parallel_{C^{0}}+\sum_{i=1}^{2}\parallel F_{i}\parallel_{C^{0,\alpha}}+\sum_{i=1}^{2}\parallel \Phi_{i}\parallel_{C^{1,\alpha}})
$$
{\bf Step 2}: We show $\parallel \partial_{q}^{n}h\parallel_{C^{2,\alpha}}\leq C^{n}n!$, where $C>0$ is independent on $n$.

Indeed, the estimates (Lemma 3.5-3.8) in \cite{9} show that there is a constant $C(\alpha,\|h\|_{C^{4,\alpha}},\|\beta\|_{C^{0,\alpha}},\|\rho\|_{C^{1,\alpha}})$ such that
\begin{eqnarray}
\parallel \partial_{q}^{n}h\parallel_{C^{2,\alpha}}\leq C^{n}n! \label{eq4.5}
\end{eqnarray}
provided that the new term $g\rho'\sum_{k=0}^{n}C_{n}^{k}(\partial_{q}^{k}h)
(\partial_{q}^{n-k}h_{p}^{3})$ in $F_{2}$ satisfies
\begin{eqnarray}
\parallel g\rho'\sum_{k=0}^{n}C_{n}^{k}(\partial_{q}^{k}h)
(\partial_{q}^{n-k}h_{p}^{3})\parallel_{C^{0,\alpha}}\leq C_{1}^{n}n! \label{eq4.6}
\end{eqnarray}
where $C_{1}$, independent on $n$, is a constant.
Because of $\rho\in C^{1,\alpha}([p_{0},0];R^{+})$, $\beta\in C^{0,\alpha}([0,|p_{0}|];R)$, and
$$\parallel \sum_{k=0}^{n}C_{n}^{k}(\partial_{q}^{k}h)
(\partial_{q}^{n-k}h_{p}^{3})\parallel_{C^{0,\alpha}}\leq\sum_{k=0}^{n}C_{n}^{k}
\parallel\partial_{q}^{k}h\parallel_{C^{0,\alpha}}\parallel\partial_{q}^{n-k}h_{p}^{3}\parallel_{C^{0,\alpha}},$$
then we can apply the skillful Lemma \ref{lemma4.1}, with $l=2,k_{0}=n,H=C_{1},u_{1}=u_{2}=u_{3}=h_{p}$, to get the result.\\
{\bf Step 3}: The mapping $q\mapsto h(q,p)$ is analytic in $R$.

From Step 2, we have
\begin{eqnarray}
max_{(q,p)\in\overline{D}}|\partial_{q}^{n}h(q,p)|\leq\parallel\partial_{q}^{n}h\parallel_{C^{2,\alpha}}\leq C^{n}n! \label{eq4.7}
\end{eqnarray}
That is to say, for any $p\in[p_{0},0]$,
\begin{eqnarray}
max_{q\in R}|\partial_{q}^{n}h(q,p)|\leq C^{n}n!,~~~~for ~~ n\in N \label{eq4.8}
\end{eqnarray}
due to periodicity. Thus we can write the $h(q,p)$ the form of Taylor series at some point $(q_{0},p)\in \overline{D}$, and the remainder term is
\begin{eqnarray}
\frac{\partial_{q}^{n+1}h(q_{0},p)}{(n+1)!}(q-q_{0})^{n+1}. \label{eq4.9}
\end{eqnarray}
According to (\ref{eq4.6}), then
\begin{eqnarray}
|\frac{\partial_{q}^{n+1}h(q_{0},p)}{(n+1)!}(q-q_{0})^{n+1}|\leq
|\frac{q-q_{0}}{\frac{1}{C}}|^{n+1}\rightarrow0, ~n\rightarrow\infty \label{eq4.10}
\end{eqnarray}
when $q_{0}-\frac{1}{C}<q<q_{0}+\frac{1}{C}$. Now we finish the proof.

\end{proof}
\begin{proposition}\label{proposition4.1}
If the mapping $q\mapsto h(q,p)$ is analytic in $R$, then each streamline $y=y(x)$ is a real-analytic curve. Moreover, the mapping $x\mapsto \psi(x,y)$ is analytic in $R$.
\end{proposition}
\begin{proof}
The first statement is easy to be proved. Indeed, each streamline $\psi(x,y)=p$, with fixed $p\in[p_{0},0]$, can be described by the graph of some function $y=y(x)$ ($y=\eta(x)$ if $p=0$) due to (\ref{eq2.7}). The analyticity of $x\mapsto y(x)$ follows at once from the analyticity of $q\mapsto h(q,p)$ in Lemma \ref{lemma4.2} because of the partial hodograph change of variables (\ref{eq2.15}).
At the same time, the explicit analytic function $y=y(x)$ is determined by the implicit function $\psi(x,y)=p$ where $p$ is a fixed constant. Therefore, the analyticity of $x\mapsto \psi(x,y)$ can be obtained due to the analyticity of $y(x)$. (see \cite{25})
\end{proof}

Now we turn to the proof of our recovery Theorem \ref{theorem4.1}.
\begin{proof}
Proposition \ref{proposition4.1} implies for each point $(x_{0},y)\in\overline{D}$ that
\begin{eqnarray}
|\psi(x,y)-\sum_{k=0}^{n}\frac{\partial_{q}^{k}\psi(x_{0},y)}{k!}(x-x_{0})^{k}|\leq\varepsilon,~n\rightarrow\infty \label{eq4.11}
\end{eqnarray}
if $|x-x_{0}|<\delta$. Because analyticity is local property, the radius of convergence of above Taylor series is independent of the point $(x_{0},y)$. So we can extend the function $\psi$ to the whole area $\overline{D}$. Without loss of generality, here we just need to determine $\psi(x,y)$ in $\{(x,y)| -\delta\leq x\leq\delta, -d\leq y\leq\eta(x)\}$, where $\delta$ is enough small. According to the symmetry and evenness of $\psi$ about $x$, we know that
\begin{eqnarray}
\partial_{x}^{2n+1}\psi(0,y)=0,~~for~n\in N,~y\in [-d,\eta(x)].
\label{eq4.12}
\end{eqnarray}
Therefore, the Taylor series of $\psi(x,y)$, at $x=0$, is
\begin{eqnarray}
\psi(x,y)=a_{2n}(y)x^{2n},~~for~|x|<\delta,~y\in [-d,\eta(x)].
\label{eq4.13}
\end{eqnarray}
where $a_{2n}(y)=\sum_{n=0}^{+\infty}\frac{\partial_{x}^{2n}\psi(0,y)}{(2n)!}$.

Then we can determine all coefficients $a_{2n}(y)$ according to our assumptions in theorem. Indeed, (\ref{eq2.8}) and (\ref{eq2.4}) imply
\begin{align}
\frac{dp_{0}}{dx}
&=\eta_{x}(\sqrt{\rho(x,\eta(x))}[u(x,\eta(x))-c])+\int_{-d}^{\eta(x)}\partial_{x}(\sqrt {\rho(x,y)}[u(x,y)-c])dy \nonumber\\
&=\sqrt{\rho(x,\eta(x))}v(x,\eta(x))-\int_{-d}^{\eta(x)}\partial_{y}(\sqrt {\rho(x,y)}v)dy=0,
\label{eq4.14}
\end{align}
that is to say, $p_{0}$ has nothing to do with $x$. Therefore, choosing $x=0$,
\begin{eqnarray}
p_{0}= \int_{-d}^{\eta(0)}\sqrt {\rho(0,y)}[u(0,y)-c]dy,
\label{eq4.15}
\end{eqnarray}
is determined by our assumptions.
The definition of pseudo-stream function $\psi(x,y)$ implies
\begin{eqnarray}
\psi(x,y)=-p_{0}+\int_{-d}^{y}\sqrt {\rho(x,s)}[u(x,s)-c]ds
\label{eq4.16}
\end{eqnarray}
Thus,
\begin{eqnarray}
a_{0}(y)=\psi(0,y)=-p_{0}+\int_{-d}^{y}\sqrt {\rho(0,s)}[u(0,s)-c]ds
\label{eq4.17}
\end{eqnarray}
is determined.

Taking the Taylor expansions (\ref{eq4.13}) of $\psi$ into the first equation of
(\ref{eq2.14}), we have
\begin{eqnarray}
\sum_{n=1}^{\infty}(2n)(2n-1)\frac{\partial_{x}^{2n}\psi(0,y)}{(2n)!}x^{2n-2}+
\sum_{n=0}^{\infty}\frac{\partial_{y}^{2}\partial_{x}^{2n}\psi(0,y)}{(2n)!}x^{2n}=gy\rho'-\beta.
\label{eq4.18}
\end{eqnarray}
Moreover, the given functions $\rho(-\psi)$ and $\beta(\psi)$ are analytic, so it's not difficult to write $\rho'$ and $\beta$ in following form
\begin{eqnarray}
\rho'=\sum_{n=0}^{\infty}b_{2n}(y)x^{2n},\label{eq4.19}\\
\beta=\sum_{n=0}^{\infty}c_{2n}(y)x^{2n},\label{eq4.20}
\end{eqnarray}
due to (\ref{eq4.13}) in $\{(x,y)| -\delta\leq x\leq\delta, -d\leq y\leq\eta(x)\}$.
Therefore, (\ref{eq4.18}) (\ref{eq4.19}) and (\ref{eq4.20}) imply
\begin{eqnarray}
\left\{\begin{array}{ll}
{ 2a_{2}(y)+a_{0}''(y)=gyb_{0}(y)+c_{0}(y), } \\
{12a_{4}(y)+a_{2}''(y)=gyb_{2}(y)+c_{2}(y),} \\
{...} \\
{(2n)(2n-1)a_{2n}(y)+a_{2n-2}''(y)=gyb_{2n-2}(y)+c_{2n-2}(y).}\end{array}\right. \label{eq4.21}
\end{eqnarray}
From the first equation in (\ref{eq4.21}), $a_{2}(y)$ will be determined. Then, $a_{4}(y)$ will be confirmed according to the second equation in (\ref{eq4.21}). All the coefficients $a_{2n}(y)$ would be obtained by repeating above process, that is to say, the pseudo-stream function $\psi=\psi(x,y)$ is determined. By the way, this iteration process can be easily achieved by computer simulation, thus our thinking would be useful in numerical calculation.

Finally, to finish our proof, it's necessary for us to determine surface profile $\eta(x)$.
From the second equation of (\ref{eq2.14}), we know that
\begin{eqnarray}
\eta(x)=\frac{Q-|\nabla\psi|^{2}}{2g\rho|_{\eta}}-d,\label{eq4.22}
\end{eqnarray}
where
\begin{eqnarray}
Q=2(E|_{\eta}-P_{atm}+g\rho |_{\eta}d)=2(\frac{(u(0,\eta(0))-c)^{2}}{2}+g\rho|_{\eta}\eta(0)+g\rho |_{\eta}d).\label{eq4.23}
\end{eqnarray}
From (\ref{eq4.22}) and (\ref{eq4.23}), we have
\begin{eqnarray}
\eta(x)=\frac{(u(0,\eta(0))-c)^{2}-|\nabla\psi|^{2}}{2g\rho|_{\eta}}+\eta(0),\label{eq4.24}
\end{eqnarray}
Thus, the free surface $\eta(x)$ is determined by our assumptions and (\ref{eq4.24}). Up to now, we finish our proof.
\end{proof}
\section*{Acknowledgement}
The authors acknowledge the support of the National Natural Science Foundation of China (No.11571057).

\section*{References}

\bibliography{mybibfile}

\end{document}